\documentclass[12pt,centertags,oneside]{amsart}
\usepackage{amsmath,amstext,amsthm,amscd,typearea,hyperref}
\usepackage{amssymb}

\usepackage{amscd,amsxtra,calc}
\usepackage{cmmib57}
\usepackage{url}

\usepackage{amscd}
\usepackage{pstricks}
\usepackage{color}
\usepackage{ulem}
\usepackage[a4paper,width=16.2cm,top=3cm,bottom=3cm]{geometry}

\numberwithin{equation}{section}

\theoremstyle{plain}
\newtheorem{thm}{Theorem}[section]
\newtheorem{theorem}[thm]{Theorem}

\newtheorem{lemma}[thm]{Lemma}

\theoremstyle{definition}

\numberwithin{equation}{section}

\newcommand{\Pic}{{\rm Pic}}

\newcommand{\Spec}{{\rm Spec \,}}

\newcommand{\sA}{{\mathcal A}}

\newcommand{\sQ}{{\mathcal Q}}

\newcommand{\C}{{\mathbb C}}

\renewcommand{\P}{{\mathbb P}}

\newcommand{\R}{{\mathbb R}}

\newcommand{\Z}{{\mathbb Z}}

\newcommand{\id}{{\rm id\hspace{.1ex}}}

\newcommand{\Aut}{{\rm Aut\hspace{.1ex}}}

\newcommand{\Ine}{{\rm Ine}}

\title [Rational surfaces with infinitely many real forms]{Smooth complex projective rational surfaces with infinitely many real forms}

\author{Tien-Cuong Dinh}
\address{Department of Mathematics, National University of Singapore, 10, 
Lower Kent Ridge Road, Singapore 119076}
\email{matdtc@nus.edu.sg}
\author{Keiji Oguiso}
\address{Mathematical Sciences, the University of Tokyo, Meguro Komaba 3-8-1, Tokyo, Japan, and National Center for Theoretical Sciences, Mathematics Division, National Taiwan University, 
Taipei, Taiwan}
\email{oguiso@ms.u-tokyo.ac.jp}
\author{Xun Yu}
\address{Center for Applied Mathematics, Tianjin University, 92 Weijin Road, Nankai District,
Tianjin 300072, P. R. China.
}
\email{xunyu@tju.edu.cn}
\thanks{The first named author is supported by the NUS grant R-146-000-319-114 and MOE grant MOE-T2EP20120-0010. The second named author is supported by JSPS Grant-in-Aid 20H00111, 20H01809, and by NCTS Scholar Program. The third named author is supported by NSFC (No.12071337, No. 11701413, and No. 11831013).}


\begin{document}

\maketitle

\begin{abstract}
We construct a smooth complex projective rational surface with infinitely many mutually non-isomorphic real forms. This gives the first definite answer to a long standing open question if a smooth complex projective rational surface has only finitely many non-isomorphic real forms or not.
\end{abstract}

\section{Introduction}

We work over $\C$. Let $V_{\R} \to {\rm Spec}\, \R$ be an $\R$-scheme. The associated $\C$-scheme
$$V = V_{\R} \times_{{\rm Spec}\, \R} {\rm Spec}\, \C$$ 
is said to be defined over $\R$. We call $U_{\R} \to {\Spec}\, \R$ a real form, or a real structure, of $V$ if 
$U_{\R} \times_{{\rm Spec}\, \R} {\rm Spec}\, \C$ is isomorphic to $V$ over ${\rm Spec}\, \C$. Two real forms $U_{\R} \to {\rm Spec}\, \R$ and $U_{\R}^\prime \to {\rm Spec}\, \R$ of $V$ are said to be isomorphic if they are isomorphic over ${\rm Spec}\, \R$. (See \cite{BS64} and \cite[Chapter 5]{Se02} for details.) In this paper, we denote the set of real points $V_{\R}(\R)$ simply by $V(\R)$ when $V_{\R}$ is fixed. 

\medskip

The aim of this paper is to give the first definite answer (Theorem \ref{MainIntro}) to the long standing, notoriously difficult question:{\it "Is there a smooth complex projective rational surface with infinitely many mutually non-isomorphic real forms\,?"} (\cite{Kh02}, \cite{Le18}, \cite{DO19}, see also \cite{DK02}, \cite{DIK04}, \cite{Be16}, \cite{Be17}, \cite{DFMJ21}, \cite{Ki20}, \cite{Le21}, \cite{DOY22}, \cite{Bo21a}, \cite{Bo21b} and so on for closely related works.) Previously known examples of projective varieties with infinitely many real forms either have dimension $\ge 3$ or have Kodaira dimension $\ge 0$, and in this paper we give the {\it first} examples of projective rational surfaces with this property.

\medskip

To state our result precisely, we prepare a few notations. 
Let 
$$\P_{\R} := \P_{\R}^1 \times_{{\rm Spec}\, \R} \P_{\R}^1.$$
Then the smooth complex projective rational surface $\P := \P^1 \times \P^1$ is defined over $\R$ as
$$\P = \P_{\R} \times_{{\rm Spec}\, \R} {\rm Spec}\, \C.$$

Let $\lambda \in \R$ be a real number and $T \to \P$ be the blow up of $16$ real points $(a, b)$ of $\P(\R)$, where 
$$a \in \{0, 1, 2, \infty \}\,\, ,\,\, b \in \{0, 1, \lambda, \infty\}.$$
Let $C_{11, T} \simeq \P^1$ be the exceptional curve on $T$ over $(\infty, \infty)$. Let $Q_{T} \in C_{11, T}(\R) = \P^1(\R)$ be a real point and let 
$$X := X_{Q_T} \to T$$ 
be the blow up of $T$ at $Q_T$. Then the surface $X$ is a smooth complex projective rational surface  defined over $\R$. Note that the surface $X$ depends on the two parameters $\lambda \in \R$ and $Q_T \in C_{11, T}(\R)$. 

Our main theorem is now stated as follows:

\begin{theorem}\label{MainIntro}
If both $\lambda \in \R$ and $Q_T \in C_{11, T}(\R)$ are generic, then the smooth complex projective rational surface $X = X_{Q_T}$ has infinitely many mutually non-isomorphic real forms. 
\end{theorem}

The precise definition of the genericity of $\lambda \in \R$ is given in Section \ref{sect2} and the precise definition of the genericity of $Q_T \in C_{11, T}(\R)$ is given in Section \ref{sect3}.

\medskip

We use results in \cite{Le18} and \cite{DO19} in our proof. Especially, we prove Theorem \ref{MainIntro} by using a special K3 surface studied in \cite{Og89} and \cite{DO19}. We emphasize that our way of use of a special K3 surface in the study of real forms of rational surfaces is new and highly non-trivial as the actual main theorem (Theorem \ref{RealMain}) shows.

We know that $\Aut (X_{Q_T})$ is discrete (Lemma \ref{Picard}(2)) but we do not know if $\Aut (X_{Q_T})$ is finitely generated or not for our surface $X_{Q_T}$. It would be interesting to see {\it "if the existence of infinitely many real forms on $X_{Q_T}$ implies that $\Aut (X_{Q_T})$ is not finitely generated or not"}. We note that the discrete automorphism group $\pi_0 (\Aut V) := \Aut (V)/\Aut^0 (V)$ is not finitely generated for all {\it previously known} smooth complex projective varieties $V$ with infinitely many real forms (\cite{Le18}, \cite{DO19}, \cite{DOY22}). Here, $\Aut^0(V)$ denotes the  identity component of $\Aut (V)$.

\medskip

It is known that $V$ has only finitely many real forms and the automorphism group $\Aut (V)$ is finitely generated if $V$ is a K3 surface (\cite{DIK00}, \cite{St85}, see also \cite{CF20}). As a byproduct of our argument used in the proof of Theorem \ref{MainIntro}, we also show the following result. 

\begin{theorem}\label{Main2Intro}
There exist a smooth K3 surface $S$ and a point $Q \in S$ such that the blow up $S'$ of $S$ at $Q$ satisfies the following two statements:
\begin{enumerate}
\item $\Aut (S')$ is discrete and not finitely generated; and
\item $S'$ admits infinitely many mutually non-isomorphic real forms.
\end{enumerate}
\end{theorem}

See also Section \ref{sect4} for a more precise construction of $S'$. 

Theorem \ref{Main2Intro} gives a positive answer for the question of Mukai to us: {\it "Is it possible to make a smooth complex projective surface with non-finitely generated automorphisms group and/or with infinitely many real forms, as one point blow up of a K3 surface\,?"} and gives simplest known surfaces satisfying the two above statements. Compare with our previous construction in \cite{DO19}. 

\medskip

Our proof of Theorems \ref{MainIntro} and \ref{Main2Intro} uses the uncountability of the base fields $\R$ and $\C$ in an essential way. Especially, our proof does not tell if there exist a smooth K3 surface $S$ and a point $Q \in S$ such that Theorem \ref{Main2Intro}(1) holds when the algebraically closed base field is countable. Compare with \cite[Theorem 1.4]{Og20}.

\medskip

In Section \ref{sect2}, we fix notations concerning a Kummer surface ${\rm Km}\, (E \times F)$ and review some properties of the surface from \cite{DO19}, which we will use throughout this paper. Some of the results in Section \ref{sect2} are also new. We then prove Theorem \ref{MainIntro} in Section \ref{sect3} and Theorem \ref{Main2Intro} in Section \ref{sect4} respectively. 

\section{Preliminaries}\label{sect2}

In this section, we fix notations concerning automorphisms, some Kummer 
surface ${\rm Km}\, (E \times F)$ and also recall some properties of the surface. Though we will use the same notations in \cite{DO19}, we recall all notations we will use later for the convenience of the reader. We also recall some properties which we will use. We refer the reader to the corresponding statements of \cite{DO19} for details, while some statements presented here are new, for which we will also give a full proof.  

\medskip

For a variety $V$ and a family of subsets $W_i \subset V$ ($i \in I$), we define the subgroups $\Aut (V, W_i\, (i \in I))$, ${\rm Ine} (Y, W_i\, (i \in I))$ of $\Aut (V)$ by 
$$\Aut (V, W_i\, (i \in I)) := \{f \in \Aut (V)\,|\, f(W_i) = W_i\, (\forall i \in I)\}.$$
$$\Ine (V, W_i\, (i \in I)):=\{g\in \Aut (V, W_i (i \in I))\, |\,  g|_{W_i}=\id_{W_i}\, (\forall i \in I) \}.$$ 
We denote the fixed point set of $f \in \Aut (V)$ by
$$V^{f} := \{x \in V\, |\, f(x) = x\}.$$

\medskip

Let $E$ and $F$ be the elliptic curves defined over $\R$ respectively by the Weierstrass forms
$$y^2 = x(x-1)(x-2)\,\, ,\,\, (y')^2 = x'(x'-1)(x'- \lambda),$$
where $\lambda$ is a {\it generic} real number in the sense that $E$ and $F$ are not isogenous. For instance transcendental real number $\lambda$ is such a number. Let
$$S := {\rm Km}\, (E \times F)$$
be the Kummer K3 surface associated to the abelian surface 
$E \times F$. The surface $S$ is also defined over $\R$ under the natural real structure induced from the real structures of $E$ and $F$ above.

Let $\{a_i\}_{i=1}^{4}$ and $\{b_i\}_{i=1}^{4}$ be the $2$-torsion groups of $F$ and $E$ respectively. Then $S$ contains 24 smooth rational curves as in Figure 1; $8$ smooth rational curves $E_i$, $F_i$, with $1 \le i \le 4$, arising from $8$ elliptic curves $E \times \{a_i\}$, $\{b_i\} \times F$ on $E \times F$, and $16$ exceptional curves $C_{ij}$, with $1\le i,j \le 4$, over the $16$ singular points of type $A_1$ on the quotient surface $E \times F/\langle -1_{E \times F}\rangle$. See also Figure 1 for the configuration of these $24$ smooth rational curves on $S$. Note that these $24$ smooth rational curves are also defined over $\R$ under the natural real structure of $S$. In what follows, we use the notation in Figure 1.

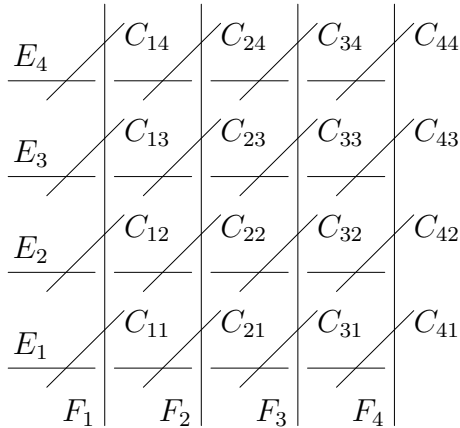
\begin{figure}
\unitlength 0.1in
\begin{picture}(25.000000,24.000000)(-1.000000,-23.500000)
\put(4.500000, -22.000000){\makebox(0,0)[rb]{$F_1$}}%
\put(9.500000, -22.000000){\makebox(0,0)[rb]{$F_2$}}%
\put(14.500000, -22.000000){\makebox(0,0)[rb]{$F_3$}}%
\put(19.500000, -22.000000){\makebox(0,0)[rb]{$F_4$}}%
\put(0.250000, -18.500000){\makebox(0,0)[lb]{$E_1$}}%
\put(0.250000, -13.500000){\makebox(0,0)[lb]{$E_2$}}%
\put(0.250000, -8.500000){\makebox(0,0)[lb]{$E_3$}}%
\put(0.250000, -3.500000){\makebox(0,0)[lb]{$E_4$}}%
\put(6.000000, -16.000000){\makebox(0,0)[lt]{$C_{11}$}}%
\put(6.000000, -11.000000){\makebox(0,0)[lt]{$C_{12}$}}%
\put(6.000000, -6.000000){\makebox(0,0)[lt]{$C_{13}$}}%
\put(6.000000, -1.000000){\makebox(0,0)[lt]{$C_{14}$}}%
\put(11.000000, -16.000000){\makebox(0,0)[lt]{$C_{21}$}}%
\put(11.000000, -11.000000){\makebox(0,0)[lt]{$C_{22}$}}%
\put(11.000000, -6.000000){\makebox(0,0)[lt]{$C_{23}$}}%
\put(11.000000, -1.000000){\makebox(0,0)[lt]{$C_{24}$}}%
\put(16.000000, -16.000000){\makebox(0,0)[lt]{$C_{31}$}}%
\put(16.000000, -11.000000){\makebox(0,0)[lt]{$C_{32}$}}%
\put(16.000000, -6.000000){\makebox(0,0)[lt]{$C_{33}$}}%
\put(16.000000, -1.000000){\makebox(0,0)[lt]{$C_{34}$}}%
\put(21.000000, -16.000000){\makebox(0,0)[lt]{$C_{41}$}}%
\put(21.000000, -11.000000){\makebox(0,0)[lt]{$C_{42}$}}%
\put(21.000000, -6.000000){\makebox(0,0)[lt]{$C_{43}$}}%
\put(21.000000, -1.000000){\makebox(0,0)[lt]{$C_{44}$}}%
\special{pa 500 2200}%
\special{pa 500 0}%
\special{fp}%
\special{pa 1000 2200}%
\special{pa 1000 0}%
\special{fp}%
\special{pa 1500 2200}%
\special{pa 1500 0}%
\special{fp}%
\special{pa 2000 2200}%
\special{pa 2000 0}%
\special{fp}%
\special{pa 0 1900}%
\special{pa 450 1900}%
\special{fp}%
\special{pa 550 1900}%
\special{pa 950 1900}%
\special{fp}%
\special{pa 1050 1900}%
\special{pa 1450 1900}%
\special{fp}%
\special{pa 1550 1900}%
\special{pa 1950 1900}%
\special{fp}%
\special{pa 0 1400}%
\special{pa 450 1400}%
\special{fp}%
\special{pa 550 1400}%
\special{pa 950 1400}%
\special{fp}%
\special{pa 1050 1400}%
\special{pa 1450 1400}%
\special{fp}%
\special{pa 1550 1400}%
\special{pa 1950 1400}%
\special{fp}%
\special{pa 0 900}%
\special{pa 450 900}%
\special{fp}%
\special{pa 550 900}%
\special{pa 950 900}%
\special{fp}%
\special{pa 1050 900}%
\special{pa 1450 900}%
\special{fp}%
\special{pa 1550 900}%
\special{pa 1950 900}%
\special{fp}%
\special{pa 0 400}%
\special{pa 450 400}%
\special{fp}%
\special{pa 550 400}%
\special{pa 950 400}%
\special{fp}%
\special{pa 1050 400}%
\special{pa 1450 400}%
\special{fp}%
\special{pa 1550 400}%
\special{pa 1950 400}%
\special{fp}%
\special{pa 200 2000}%
\special{pa 600 1600}%
\special{fp}%
\special{pa 200 1500}%
\special{pa 600 1100}%
\special{fp}%
\special{pa 200 1000}%
\special{pa 600 600}%
\special{fp}%
\special{pa 200 500}%
\special{pa 600 100}%
\special{fp}%
\special{pa 700 2000}%
\special{pa 1100 1600}%
\special{fp}%
\special{pa 700 1500}%
\special{pa 1100 1100}%
\special{fp}%
\special{pa 700 1000}%
\special{pa 1100 600}%
\special{fp}%
\special{pa 700 500}%
\special{pa 1100 100}%
\special{fp}%
\special{pa 1200 2000}%
\special{pa 1600 1600}%
\special{fp}%
\special{pa 1200 1500}%
\special{pa 1600 1100}%
\special{fp}%
\special{pa 1200 1000}%
\special{pa 1600 600}%
\special{fp}%
\special{pa 1200 500}%
\special{pa 1600 100}%
\special{fp}%
\special{pa 1700 2000}%
\special{pa 2100 1600}%
\special{fp}%
\special{pa 1700 1500}%
\special{pa 2100 1100}%
\special{fp}%
\special{pa 1700 1000}%
\special{pa 2100 600}%
\special{fp}%
\special{pa 1700 500}%
\special{pa 2100 100}%
\special{fp}%
\end{picture}%
 \caption{Curves $E_i$, $F_j$ and $C_{ij}$ with $E_i^2=F_j^2=C_{ij}^2=-2$}
 \label{fig1}
\end{figure}

Throughout this paper, we use $x$ for the affine coordinate of $E_1 = E/\langle -1_E \rangle$ and $x'$ for the affine coordinate of $F_1 = F/\langle -1_F \rangle$ and set (with respect to the coordinates $x$ and $x'$):
$$C := E_{1}\,\, ,\,\, P := C \cap C_{11} = \infty\,\, ,\,\, C \cap C_{21} = 0\,\, ,\,\, C \cap C_{31} = 1\,\, ,\,\, C \cap C_{41} = 2.$$
$$P_1 := F_1 \cap C_{11} = \infty\,\, ,\,\, F_1 \cap C_{12} = 0\,\, ,\,\, F_1 \cap C_{13} = 1\,\, ,\,\, F_1 \cap C_{14} = \lambda.$$

Let $\theta \in \Aut (S)$ be the involution induced by the involution $(1_E, -1_F) \in \Aut (E \times F)$. Then $\theta^* \omega_S = -\omega_S$. Here and hereafter $\omega_S$ stands for a non-zero global holomorphic $2$-form on $S$, which is unique up to scalar multiplications by $\C \setminus \{0\}$. 

We consider the quotient surface and the associated quotient morphism 
$$T := S/\langle \theta \rangle\,\, ,\,\, \pi : S \to T.$$
By definition of $\theta$, we find that $T$ is the blow-up of $\P^1 \times \P^1$ at the $16$ points $(a, b)$ ($a \in \{\infty, 0, 1, 2\}$, $b \in \{\infty, 0, 1, \lambda\}$) under the identification $\P^1 \times \P^1 = C \times F_1$. In particular, $T$ is a smooth complex projective rational surface defined over $\R$. Note also that the surface $T$ with natural real structure here coincides with the surface $T$ with real structure defined in Introduction.

\begin{lemma}\label{theta} We have the following:
\begin{enumerate}
\item $S^{\theta} = \cup_{i=1}^{4} (E_i \cup F_i)$ and $f \circ \theta = \theta \circ f$ for all $f \in {\rm Aut}\, (S)$. In particular, $\Aut (S) = \Aut (S,\cup_{i=1}^{4} (E_i \cup F_i))$. Moreover, $\theta(R) = R$ for any smooth rational curve $R$ on $S$.
\item $\Aut (S, P) = \Aut(S, C, P)$.
\item $\Aut (T) = \Aut (S)/ \langle \theta \rangle$. 
\end{enumerate}
\end{lemma}

\begin{proof} The assertion (1) is proved by \cite[Lemma 3.3]{DO19} or by \cite[Lemmas (1.3) and (1.4)]{Og89}. The assertion (2) is proved by \cite[Lemma 3.4]{DO19} as an immediate but very important consequence of (1). Indeed, since $C = E_1$ is the unique irreducible component of $\cup_{i=1}^{4} (E_i \cup F_i)$ containing the point $P$, we find that $f(C) = C$ if $f \in \Aut (S, P)$. 

We show the assertion  (3). By (1), we have a natural inclusion $\Aut (S)/ \langle \theta \rangle \subset \Aut (T)$. By the construction of $T$, the quotient morphism $\pi$ is nothing but the finite double cover branched along the divisor $D := \sum_{i=1}^{4} (\pi(E_i) + \pi(F_i))$. Note that $D$ is a disjoint union of $8$ smooth rational curves of self-intersection $-4$ and $D \in |-2K_T|$. We claim that 
$$|-2K_T| = \{D\}.$$
This is proved as follows. Observe that $D$ is a sum of disjoint irreducible curves with negative self-intersection number. If $\dim |D| \ge 1$, then $|D|$ has a fixed component, as $(D^2)_T < 0$. Let $R$ be an irreducible component of the fixed component of $|D|$. Then $R$ is one of $\pi(E_i)$ or $\pi(F_i)$, and $D-R$ is an effective divisor which is again a disjoint union of irreducible curves of negative self-intersection and satisfies $\dim |D-R| = \dim |D| \ge 1$. Then, applying the same argument as above, we deduce that $|D-R|$ has a fixed component, say $R'$, such that $D-R -R'$ is an effective divisor which is again a disjoint union of irreducible curves of negative self-intersection and satisfies $\dim |D-R-R'| = \dim |D-R| \ge 1$. Repeating this, we finally reach the case where $\dim |Z| \ge 1$ for 
the divisor $Z=0$, a contradiction. Thus, $\dim |D| = 0$ and the claim follows.

Let $f \in \Aut (T)$. Then, since $f$ preserves the class of $-2K_T$ in $\Pic\, (T) \simeq {\rm NS}\, (T)$, we have 
$$f \in \Aut\big(T,\sum_{i=1}^{4} (\pi(E_i) + \pi(F_i))\big)$$
by the claim above. Hence $f$ lifts to an automorphism of $S$. This completes the proof.  
\end{proof}

From now, we denote $\pi(W)$ simply by $W_T$. For instance $\pi(C_{11}) = C_{11, T}$, which is consistent with the notation in Introduction. 

\medskip 

Next, we recall one more important involution $\iota$ of $S$ and 
its conjugates from \cite{DO19}.
As in \cite[Page 956]{DO19}, consider the elliptic fibrations $\Phi_{D_i} : S \to \P^1$ ($i = 1$, $2$) on $S$ defined respectively by the complete linear systems $|D_i|$ of the divisors of Kodaira's singular fiber type,
$$D_1 := C + C_{11} + F_1 + C_{12} + E_2 + C_{22} + F_2 + C_{21}$$
and
$$D_2 := C + 2C_{11} + E_2 + 2C_{12} + E_{3} + 2C_{13} + 3F_{1},$$
see Figure \ref{fig2}. We choose $C_{31}$ as the zero section of both $\Phi_{D_1}$ and $\Phi_{D_2}$. 

\begin{figure}
		\unitlength 0.1in
		\begin{picture}(25.000000,24.000000)(-1.000000,-23.500000)
		
			\put(-10.500000, -22.000000){\makebox(0,0)[rb]{$F_1$}}%
			\put(-5.500000, -22.000000){\makebox(0,0)[rb]{$F_2$}}%
			\put(-0.500000, -22.000000){\makebox(0,0)[rb]{$F_3$}}%
			\put(4.500000, -22.000000){\makebox(0,0)[rb]{$F_4$}}%
			\put(-14.750000, -18.500000){\makebox(0,0)[lb]{$C$}}%
			\put(-14.750000, -13.500000){\makebox(0,0)[lb]{$E_2$}}%
			\put(-14.750000, -8.500000){\makebox(0,0)[lb]{$E_3$}}%
			\put(-14.750000, -3.500000){\makebox(0,0)[lb]{$E_4$}}%
			\put(-9.000000, -16.000000){\makebox(0,0)[lt]{$C_{11}$}}%
			\put(-9.000000, -11.000000){\makebox(0,0)[lt]{$C_{12}$}}%
			\put(-9.000000, -6.000000){\makebox(0,0)[lt]{$C_{13}$}}%
			\put(-9.000000, -1.000000){\makebox(0,0)[lt]{$C_{14}$}}%
			\put(-4.000000, -16.000000){\makebox(0,0)[lt]{$C_{21}$}}%
			\put(-4.000000, -11.000000){\makebox(0,0)[lt]{$C_{22}$}}%
			\put(-4.000000, -6.000000){\makebox(0,0)[lt]{$C_{23}$}}%
			\put(-4.000000, -1.000000){\makebox(0,0)[lt]{$C_{24}$}}%
			\put(1.000000, -16.000000){\makebox(0,0)[lt]{$C_{31}$}}%
			\put(1.000000, -11.000000){\makebox(0,0)[lt]{$C_{32}$}}%
			\put(1.000000, -6.000000){\makebox(0,0)[lt]{$C_{33}$}}%
			\put(1.000000, -1.000000){\makebox(0,0)[lt]{$C_{34}$}}%
			\put(6.000000, -16.000000){\makebox(0,0)[lt]{$C_{41}$}}%
			\put(6.000000, -11.000000){\makebox(0,0)[lt]{$C_{42}$}}%
			\put(6.000000, -6.000000){\makebox(0,0)[lt]{$C_{43}$}}%
			\put(6.000000, -1.000000){\makebox(0,0)[lt]{$C_{44}$}}%
			
			\linethickness{1mm}
			\put(-10.000000, -22.000000){\line(0,1){22}}
                          \linethickness{1mm}
			 \put(-5.000000, -22.000000){\line(0,1){22}}
                          \linethickness{0.1mm}
			 \put(0.000000, -22.000000){\line(0,1){22}}
			 \linethickness{0.1mm}
			 \put(5.000000, -22.000000){\line(0,1){22}}
                          \linethickness{1mm}
			 \put(-15.000000, -19.000000){\line(1,0){4.5}}
			\linethickness{1mm}
			\put(-9.500000, -19.000000){\line(1,0){4}}
                         \linethickness{1mm}
			\put(-4.500000, -19.000000){\line(1,0){4}}
                         \linethickness{1mm}
			\put(0.500000, -19.000000){\line(1,0){4}}
			 \linethickness{1mm}
			 \put(-15.000000, -14.000000){\line(1,0){4.5}}
			 \linethickness{1mm}
			 \put(-9.500000, -14.000000){\line(1,0){4}}
			 \linethickness{1mm}
			 \put(-4.500000, -14.000000){\line(1,0){4}}
			\linethickness{1mm}
			 \put(0.500000, -14.000000){\line(1,0){4}}
			 \linethickness{0.1mm}
			 \put(-15.000000, -9.000000){\line(1,0){4.5}}
			 \linethickness{0.1mm}
			 \put(-9.500000, -9.000000){\line(1,0){4}}
			 \linethickness{0.1mm}
			 \put(-4.500000, -9.000000){\line(1,0){4}}
			 \linethickness{0.1mm}
			 \put(0.500000, -9.000000){\line(1,0){4}}
			 \linethickness{0.1mm}
			 \put(-15.000000, -4.000000){\line(1,0){4.5}}
			 \linethickness{0.1mm}
			 \put(-9.500000, -4.000000){\line(1,0){4}}
			 \linethickness{0.1mm}
			 \put(-4.500000, -4.000000){\line(1,0){4}}
			 \linethickness{0.1mm}
			 \put(0.500000, -4.000000){\line(1,0){4}}
			                
                \put(19.500000, -22.000000){\makebox(0,0)[rb]{$F_1$}}%
			\put(24.500000, -22.000000){\makebox(0,0)[rb]{$F_2$}}%
			\put(29.500000, -22.000000){\makebox(0,0)[rb]{$F_3$}}%
			\put(34.500000, -22.000000){\makebox(0,0)[rb]{$F_4$}}%
			\put(15.250000, -18.500000){\makebox(0,0)[lb]{$C$}}%
			\put(15.250000, -13.500000){\makebox(0,0)[lb]{$E_2$}}%
			\put(15.250000, -8.500000){\makebox(0,0)[lb]{$E_3$}}%
			\put(15.250000, -3.500000){\makebox(0,0)[lb]{$E_4$}}%
			\put(21.000000, -16.000000){\makebox(0,0)[lt]{$C_{11}$}}%
			\put(21.000000, -11.000000){\makebox(0,0)[lt]{$C_{12}$}}%
			\put(21.000000, -6.000000){\makebox(0,0)[lt]{$C_{13}$}}%
			\put(21.000000, -1.000000){\makebox(0,0)[lt]{$C_{14}$}}%
			\put(26.000000, -16.000000){\makebox(0,0)[lt]{$C_{21}$}}%
			\put(26.000000, -11.000000){\makebox(0,0)[lt]{$C_{22}$}}%
			\put(26.000000, -6.000000){\makebox(0,0)[lt]{$C_{23}$}}%
			\put(26.000000, -1.000000){\makebox(0,0)[lt]{$C_{24}$}}%
			\put(31.000000, -16.000000){\makebox(0,0)[lt]{$C_{31}$}}%
			\put(31.000000, -11.000000){\makebox(0,0)[lt]{$C_{32}$}}%
			\put(31.000000, -6.000000){\makebox(0,0)[lt]{$C_{33}$}}%
			\put(31.000000, -1.000000){\makebox(0,0)[lt]{$C_{34}$}}%
			\put(36.000000, -16.000000){\makebox(0,0)[lt]{$C_{41}$}}%
			\put(36.000000, -11.000000){\makebox(0,0)[lt]{$C_{42}$}}%
			\put(36.000000, -6.000000){\makebox(0,0)[lt]{$C_{43}$}}%
			\put(36.000000, -1.000000){\makebox(0,0)[lt]{$C_{44}$}}%
			
			\linethickness{1mm}
			\put(20.000000, -22.000000){\line(0,1){22}}
                          \linethickness{0.1mm}
			 \put(25.000000, -22.000000){\line(0,1){22}}
                          \linethickness{0.1mm}
			 \put(30.000000, -22.000000){\line(0,1){22}}
			 \linethickness{0.1mm}
			 \put(35.000000, -22.000000){\line(0,1){22}}
                          \linethickness{1mm}
			 \put(15.000000, -19.000000){\line(1,0){4.5}}
			\linethickness{1mm}
			\put(20.500000, -19.000000){\line(1,0){4}}
                         \linethickness{1mm}
			\put(25.500000, -19.000000){\line(1,0){4}}
                         \linethickness{1mm}
			\put(30.500000, -19.000000){\line(1,0){4}}
			 \linethickness{1mm}
			 \put(15.000000, -14.000000){\line(1,0){4.5}}
			 \linethickness{1mm}
			 \put(20.500000, -14.000000){\line(1,0){4}}
			 \linethickness{1mm}
			 \put(25.500000, -14.000000){\line(1,0){4}}
			\linethickness{1mm}
			 \put(30.500000, -14.000000){\line(1,0){4}}
			 \linethickness{1mm}
			 \put(15.000000, -9.000000){\line(1,0){4.5}}
			 \linethickness{1mm}
			 \put(20.500000, -9.000000){\line(1,0){4}}
			 \linethickness{1mm}
			 \put(25.500000, -9.000000){\line(1,0){4}}
			 \linethickness{1mm}
			 \put(30.500000, -9.000000){\line(1,0){4}}
			 \linethickness{0.1mm}
			 \put(15.000000, -4.000000){\line(1,0){4.5}}
			 \linethickness{0.1mm}
			 \put(20.500000, -4.000000){\line(1,0){4}}
			 \linethickness{0.1mm}
			 \put(25.500000, -4.000000){\line(1,0){4}}
			 \linethickness{0.1mm}
			 \put(30.500000, -4.000000){\line(1,0){4}}
			                
		\end{picture}%
		\caption{Divisors $D_1$ and $D_2$}
		\label{fig2}
	\end{figure}
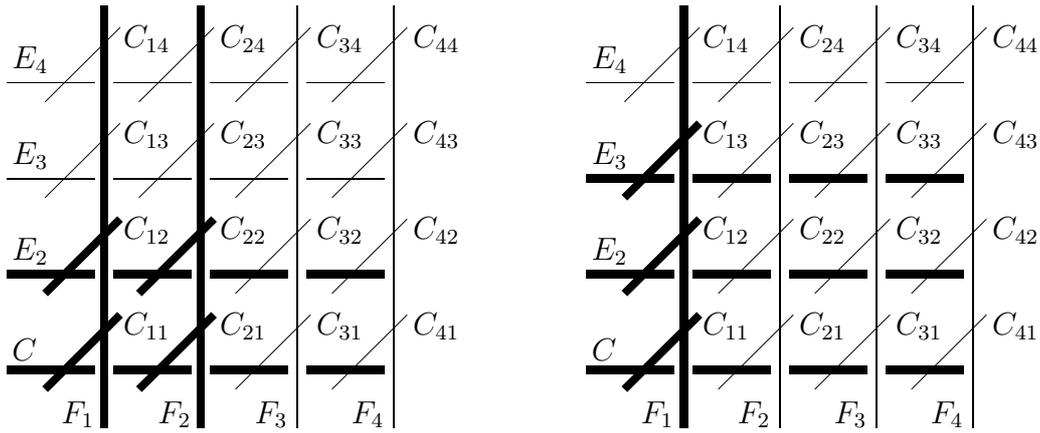

Denote by $\iota$ the inversion of the elliptic fibration $\Phi_{D_2}$ (with respect to the zero section $C_{31}$). Then $\iota \in \Aut (S)$ and $\iota$ is an involution of $S$ such that $\iota^*\omega_S = -\omega_S$.
Denote by $f_1$ the translation automorphism of the elliptic fibration $\Phi_{D_1}$ given by the section $C_{41}$ (with respect to the zero section $C_{31}$). Then $f_1 \in \Aut (S)$ and $f_1^*\omega_S = \omega_S$. 
We finally define
$$\iota_n := f_1^{-n} \circ \iota \circ f_1^{n}\,\, (\forall n \in \Z).$$

\begin{lemma}\label{iota} 
The automorphisms $\iota$, $f_1$ and $\iota_n$ satisfy the following properties:
\begin{enumerate}
\item The fixed locus $S^{\iota}$ is a disjoint union of three smooth rational curves $C_{11}$, $C_{31}$, $C_{34}$ and a smooth curve of genus $4$, with self-intersection number $6$. 
\item $\iota$ acts as a nontrivial involution on each of the smooth rational curves $C, F_1, F_3, E_4$. 
\item All $\iota, f_1$ and $\iota_n$ preserve $C$ and $C_{11}$. Moreover, $\iota_n$ fixes $C_{11}$ pointwisely and we have
$$\iota|_{C} (x) = 2-x\,\, ,\,\, f_1|_{C} (x) = 2x\,\, ,\,\, \iota_n|_{C}(x) = \frac{1}{2^{n-1}} -x.$$
\end{enumerate}
\end{lemma}
\begin{proof}
(1) This assertion  is proved by \cite[Lemmas 4.2 and 4.3]{DO19}. 

\smallskip

(2) Recall that $\iota$ is the inversion of the elliptic fibration $\Phi_{D_2}$ with respect to the zero section $C_{31}$ and $D_2$ is a fiber of this fibration. 
It follows that $D_2$ is invariant by $\iota$. Since $F_1$ is the only component with coefficient 3 in $D_2$, it is invariant by $\iota$. The component $C$ contains a fixed point which is its intersection with $C_{31}$, see (1). It is invariant as well.

Recall that the divisor
$$D_2^\prime := F_3 + 2C_{34} + F_2 + 2C_{24} + F_{4} + 2C_{44} + 3E_{4}$$
is also a fiber of $\Phi_{D_2}$, see e.g. \cite[Page 958]{DO19}. Arguing as above, we obtain that $E_4$ and $F_3$ are invariant by $\iota$.
Finally, by (1), the restriction of $\iota$ to each of the curves $C, F_1, F_3, E_4$ is a nontrivial involution. 

\smallskip

(3) We have seen that $\iota$ preserves $C$ and $C_{11}$.  The invariance of $C$  by $f_1$, as well as 
the second formula in (3), are proved in \cite[Page 957]{DO19}. It follows that the point of intersection between the two components $C$ and $C_{11}$ of $D_1$ is fixed by $f_1$. As a fiber of $\varphi_{D_1}$, $D_1$ is invariant by $f_1$. Thus, $C_{11}$ is also invariant by $f_1$. By the definition of $\iota_n$, we easily deduce that $C$ (resp. $C_{11}$) is  invariant (resp. fixed pointwisely) by $\iota_n$.

Finally, in the affine coordinate $x$ of $C$, $\iota|_C$ is a nontrivial involution which fixes the points $C\cap C_{11}=\infty$ and 
$C \cap C_{31} = 1$. Thus, $\iota|_C(x) = 2-x$ as claimed in the first formula in (3).  
The last formula follows from the first two formulas.  
\end{proof}

Let us consider the inertia group 
$$\Ine (S, C_{11})=\{g\in \Aut (S, C_{11})\, |\,  g|_{C_{11}}=\id_{C_{11}}\}$$
and its two subgroups
$$\Ine^{\pm}(S,C_{11}) := \{ g \in \Ine (S,C_{11})| \, g^*(\omega_S)= \pm \omega_S\},$$
$$\Ine^s (S,C_{11}):=\{g\in\Ine (S,C_{11})| \, g^*(\omega_S)=\omega_S\}.$$
As before, $\omega_S$ is a non-zero global holomorphic $2$-form on $S$. 

The next lemma will be also used in Sections \ref{sect3} and \ref{sect4}.

\begin{lemma}\label{ine} We have the following statements:
\begin{enumerate}
\item Let $f \in \Ine (S,C_{11})$. Then $f \in \Aut (S,C,P)$ and we have a representation 
$$\tau : \Ine (S,C_{11}) \to \Aut (C,P)\,\, ;\,\, f \mapsto f|_C.$$
The image $\tau(\Ine^s (S,C_{11}))$ of $\Ine^s (S,C_{11})$ is an abelian group.
\item The automorphisms $\iota_n$ define infinitely many conjugacy classes in $\Ine^{\pm} (S, C_{11})$.
\end{enumerate}
\end{lemma}

\begin{proof} (1) Since $P \in C_{11}$, we have $f(P) =P$ for $f \in \Ine (S,C_{11})$, that is, $f \in \Aut (S, P)$ if $f \in \Ine (S,C_{11})$. Since $\Aut (S, P) = \Aut (S, C, P)$ by Lemma \ref{theta}(2), it follows that $f \in \Aut (S,C,P)$ and we have the representation $\tau$. 

We show that $\tau(\Ine^s (S,C_{11}))$ is an abelian group. Let $f \in \Ine^s (S,C_{11})$. Since $f^*\omega_S = \omega_S$ with $f(P) = P$, we have $\det df|_{T_{S, P}} = 1$ for the differential map $df|_{T_{S, P}} : T_{S, P} \to T_{S, P}$. Since $df|_{T_{S, P}}$ preserves the transversal lines $T_{C_{11}, P}$ and $T_{C, P}$ in $T_{S, P}$, as $f$ preserves both $C$ and $C_{11}$. Moreover, $df|_{T_{S, P}}$ is identity on $T_{C_{11}, P}$ as $f \in \Ine (S,C_{11})$. Thus, $df|_{T_{S, P}}$ is also identity on $T_{C, P}$ as $\det df|_{T_{S, P}} =1$. Hence, $f|_C$ is of the form $f|_{C}(x) = x + a$ ($a \in \C$). Therefore, $\tau(\Ine^s (S,C_{11}))$ is a subgroup of the additive group $(\C, +)$. In particular, $\tau(\Ine^s (S,C_{11}))$ is an abelian group. This proves (1). 

\smallskip

(2) By Lemma \ref{iota}(1)(3) and the definition of $\iota_n$, we have $\iota_{n} \in \Ine (S, C_{11})$ for all integers $n$.
On the other hand, by the definitions of $\iota$ and $f_1$, we have $\iota^*\omega_S = -\omega_S$ and $f_1^*\omega_S = \omega_S$. It follows that 
$\iota_n^*\omega_S = -\omega_S$. Therefore, $\iota_{n} \in \Ine^{\pm} (S, C_{11})$ for all integers $n$.

Now, let $S_1$ be the blow up of $S$ at $P$ and $S_2$ the blow up of $S_1$ at a finite $\iota|_{\P(T_{S, P})}$-stable subset $\sQ$ of the exceptional curve $\P(T_{S, P}) \subset S_1$. For the same argument as in (1), we deduce that $f \in \Ine^{\pm} (S, C_{11})$ 
acts on $\P(T_{S, P})$ and the action is identity if $f^*\omega_S = \omega_S$, while the action coincides with $\iota|_{\P(T_{S, P})}$ if $f^*\omega_S = -\omega_S$.
So, $\Ine^{\pm} (S, C_{11})$ naturally lifts to a subgroup of $\Aut (S_2)$.
By  \cite[Lemma 4.5]{DO19}, for a suitable choice of $\sQ$, in any subgroup of $\Aut(S_2)$ which contains the involutions $\iota_n$, the number of different conjugacy classes of $\iota_n$ is infinite. 
The assertion (2) follows easily.
\end{proof}

We close this section by the following general lemma, which will be also used in the next two sections.

\begin{lemma}\label{specialpoint}
Let $V$ be a smooth complex projective surface. Let $D$ be an irreducible curve in $V$ such that $(D^2)_V < 0$. Then there exists a countable subset $A\subset D$ such that the following statement holds: if $Q\in D\setminus A$, $g\in \Aut (V)$ and $g(Q)\in  D$, then $g\in \Aut (V,D)$;
in particular, for any $Q\in D\setminus A$, we have $\Aut(V, Q) = \Aut(V, D, Q)$.

Moreover, this holds for any smooth rational curve $D$ on any complex projective K3 surface $V$.
\end{lemma}

\begin{proof}
Let 
$$\sA:=\big\{R\, |\, R\subset V \text{ is an irreducible curve such that } R\neq D \,\, {\rm and}\,\, (R^2)_V  < 0\big\},$$ 
and let  
$$A :=\cup_{R\in \sA}(D\cap R).$$  
Note that the N\'eron-Severi group ${\rm NS} (V)$ is countable. Moreover, for any element $\alpha\in {\rm NS} (V)$ with $(\alpha^2)_V < 0$, there exists at most one irreducible curve $R$ such that $[R]=\alpha$. Thus, $\sA$ is a countable set. Then $A$ is also countable. Let $Q\in D\setminus A$. Let $g\in \Aut(V)$ with $g(Q)\in D$. Then $g^{\pm 1}(D)$ is an irreducible curve and $(g^{\pm 1}(D)^2)_V = (D^2)_V < 0$.

To complete the proof, it suffices to show $g(D)=D$. Suppose otherwise $g(D)\neq D$. Note that $g(Q)\in D\cap g(D)$ and $Q\in g^{-1}(D)\cap D$. Since $g^{-1}(D)\in \sA$, it follows that $Q\in A$, a contradiction. This proves the first statement. The second one is obvious and the last one is also clear 
because $(D^2)_V = -2$ for a smooth rational curve $D$ on a K3 surface $V$.
\end{proof}

\section{Proof of Theorem \ref{MainIntro}}\label{sect3}

In this section, we prove Theorem \ref{MainIntro}. 
We use the same notation as in Section \ref{sect2}. For instance, 
$S = {\rm Km}\, (E \times F)$, $P \in C \subset S$, $C_{ij} \subset S$ 
are the same as in Section \ref{sect2}, $\theta$, $\iota$ and $\iota_n$ 
are the involutions of $S$, and $\pi : S \to T$ 
is the quotient morphism by $\theta$ and so on. 

From now, we often denote the object on $T$ induced from an object $W$ on $S$ by $W_{T}$. For instance the curve $\pi(C_{11})$ on $T$ is denoted by $C_{11, T}$ and the automorphism of $T$ induced by $f \in \Aut (S)$ under the isomorphism 
in Lemma \ref{theta}(3) by $f_T$.   

\medskip

Recall that by Lemma \ref{ine}(2), $\iota_n \in \Ine^{\pm} (S, C_{11})$ as an automorphism of $S$.

\begin{lemma}\label{infiniteconj}
Let $m\ge 3$ and let $Q_1,\ldots,Q_m$ be $m$ distinct points in $C_{11}$. Then, the number of conjugacy classes of the involutions $\iota_n$ $(n \in \Z)$ in $\Aut (S, C_{11}, \{Q_1,\ldots,Q_m\})$ is infinite.
\end{lemma}

\begin{proof}
Since $C_{11}\cong \P^1$ and $m\ge 3$, it follows that $$[\Aut (S,C_{11}, \{Q_1,..,Q_m\}): \Ine(S,C_{11})]<\infty.$$ 
Thus, it suffices to show that the involutions $\iota_n$ define infinitely many conjugacy classes in a subgroup of $\Ine (S,C_{11})$ of finite index. However, $\Ine^{\pm} (S, C_{11})$ is such a subgroup by Lemma \ref{ine}(2). Note that $\Ine^{\pm} (S, C_{11})$ is a finite index subgroup of $\Ine(S,C_{11})$ as the canonical representation of $\Aut (S)$ has a finite image (see e.g. \cite[Theorem 14.10]{Ue75}). 
\end{proof}

We now apply Lemma \ref{specialpoint} for the pair $(V, D) = (S, C_{11})$. As in the proof of that lemma, define
$$\sA:=\big\{R\, |\, R\subset S \text{ is an irreducible curve such that } R\neq C_{11} \,\, {\rm and}\,\, (R^2)_S  < 0\big\},$$ 
and consider the countable set
$$A :=\cup_{R\in \sA}(C_{11}\cap R).$$
Recall that $C_{11}$ is defined over $\R$ and $C_{11}(\R)$ is an uncountable set. We say that a point $Q\in C_{11}(\R)$ is {\it generic} if 
$Q\not \in A\cup\{P,P_1\}$.

\begin{lemma}\label{pointsinC11}
Let $Q\in C_{11}(\R)$ be a generic point. Then
$$\Aut (S, \{Q,\theta(Q)\})=\Aut (S, C_{11}, \{Q,\theta(Q)\})=\Aut(S, C_{11}, \{P, P_1, Q,\theta(Q)\})$$
and the involution $\theta$ belongs to this group.
\end{lemma}

\begin{proof}
Observe that $A=\theta(A)$ by Lemma \ref{theta}(1). It follows that $\theta(Q)$ is also a generic point, i.e. outside $A\cup\{P,P_1\}$.
The first identity is a consequence of Lemma \ref{specialpoint}.
The second one is a consequence of Lemma \ref{theta}(1) because the union of 8 curves $S^{\theta} = \cup_{i=1}^{4} (E_i \cup F_i)$ intersects $C_{11}$ exactly at $P$ and $P_1$.
The assertion on $\theta$ is also clear because $\{Q,\theta(Q)\}$ is invariant by $\theta$. Note that the lemma still holds if we replace $\{Q,\theta(Q)\}$ by any non-empty set of generic points which is invariant by $\theta$.
\end{proof}

 {\it From now until the end of this section, we assume that $Q\in C_{11}(\R)$ as in the last lemma.}  

\medskip

Let $S_{Q} \to S$ be the blow up at $Q$ and $\theta(Q)$. Then $\theta$ lifts to an automorphism of $S_{Q}$ and we have the quotient morphism 
$$\pi^\prime : S_{Q} \to X := X_{Q_{T}} := S_{Q}/\langle \theta \rangle.$$
Then $X$ is the same surface in Theorem \ref{MainIntro} except that the choice $Q_{T} = \pi(Q) \in T$ is not yet specified. We have natural morphisms
$$X \to T \to \P = \P^1 \times \P^1,$$
where $T \to \P$ is the blow up at $16$ points $(a, b)$ as in Section \ref{sect2} and $X \to T$ is the blow up of $T$ at $Q_{T} := \pi(Q)$. 
Observe the following lemma.

\begin{lemma}\label{Picard} The surface $X$ satisfies the following statements:
\begin{enumerate}
\item 
$${\rm NS}\,(X) \simeq {\rm Pic}\, X = \Z [H_{1}] \oplus \Z [H_{2}] \oplus \oplus_{i,j} \Z [C_{ij, X}] \oplus \Z [E_{Q_T}],$$
where letting $H_i$ ($i = 1$, $2$) be the two general rulings of $\P$, which are defined over $\R$, we denote the pullback of $H_i$ on $X$ by the same letter $H_{i}$, the curve $C_{ij, X}$ is the proper transform of the curve $C_{ij, T} = \pi(C_{ij})$ on $X$ and $E_{Q_{T}}$ is the exceptional divisor of the blow up $X \to T$ at $Q_{T} = \pi(Q)$.

\item The natural representation 
$$\rho : \Aut (X) \to \Aut ({\rm NS}\,(X)) = \Aut ({\rm Pic}\,(X))$$
is injective. 
\end{enumerate}
\end{lemma}

\begin{proof} The first assertion is clear from the morphism $X \to T \to \P$.  
Let us show the second assertion. Let $g \in {\rm Ker}\, (\rho)$. 
Then $g([E_{Q_T}]) =[E_{Q_T}]$, $g([C_{ij, T}]) =  [C_{ij, T}]$ in ${\rm Pic}\,(X)$. Note that $|E_{Q_T}| = \{E_{Q_T}\}$ and $|C_{ij, X}| = \{C_{ij, X}\}$, as the curves are irreducible of negative self-intersection numbers. Therefore, $g(E_{Q_T}) =E_{Q_T}$, $g(C_{ij, X}) =  C_{ij, X}$ 
as divisors on $X$. Thus, $g$ descends to the automorphism $g_{\P} \in \Aut (\P^1 \times \P^1)$ under the blow-down $X \to \P = \P^1 \times \P^1$ 
so that the 16 points $(a_i, b_j)$ are pointwisely fixed by $g_{\P}$. Recall that
$$\Aut (\P^1 \times \P^1) = ({\rm PGL}\, (2, \C) \times  {\rm PGL}\, (2, \C)) \cdot \langle s \rangle,$$ 
where $s(x, y) := (y, x)$ for $(x,y)\in \P^1\times\P^1$. Then by $g_{\P}(a,b) = (a, b)$ for all $16$ points $(a, b)$, it follows that $g_{\P} = \id_{\P}$. Hence $g = \id_{X}$ as well.
\end{proof}

Note that the involutions $\iota_n \in \Aut (S)$ descend to the involutions $\iota_{n, T} \in \Aut (T)$ by Lemma \ref{theta}(3). Then by the choice of $Q$ and by Lemma \ref{ine}(2), $\iota_{n, T} \in \Ine (T, C_{11, T})$. Thus, the involutions $\iota_{n, T}$ lift to the involutions on $X$, which we will denote by $\iota_{n, X}$. 
   
The actual main theorem of this paper is the following, from which Theorem \ref{MainIntro} stated in Introduction is deduced. We use here the notion of genericity introduced just before Lemma \ref{pointsinC11}.

\begin{theorem}\label{RealMain}
For every generic point $Q \in C_{11}(\R)$, the number of the conjugacy classes of the involutions $\iota_n$  $(n \in \Z)$ in the group $\Aut (X)$ is infinite. 
\end{theorem}

\begin{proof}[Proof of Theorem~\ref{RealMain} implies Theorem~\ref{MainIntro}]

We closely follow an argument of \cite{Le18}. Let $c = (\id_{X_{\R}}, c)$ be the complex conjugation of 
$$X = X_{\R} \times_{{\rm Spec}\, \R} {\rm Spec}\, \C \to {\rm Spec}\, \C$$
with respect to the natural real structure $X_{\R}$ of $X$ in the construction. As the divisors in the formula in Lemma~\ref{Picard}(1) are all defined over $\R$ with respect to the natural real structure of $X$, it follows that $c^*$ is trivial on ${\rm Pic}\, (X)$. Thus, we have $(c \circ f \circ c)^* = f^*$ on ${\rm Pic}\,(X)$ if $f\in \Aut (X)$ . Since $c \circ f \circ c \in \Aut (X)$, it follows from Lemma~\ref{Picard}(2) that $c \circ f \circ c = f$, that is, $f$ is defined over $\R$ and the conjugate action of the Galois group ${\rm Gal} (\C/\R) = \langle c \rangle$ on $\Aut (X)$ is trivial as well. Thus, by \cite[Lemma 13]{Le18} and by the fact that the conjugacy classes of $\{\iota_{n, X}\}$ under $\Aut (X)$ is infinite by Theorem~\ref{RealMain}, we conclude that $X$ admits infinitely many mutually non-isomorphic real forms.
\end{proof}

\begin{proof}[Proof of Theorem~\ref{RealMain}]

We denote by $\Sigma_0$ the unique smooth curve of genus $4$ in $S^{\iota}$  given in Lemma \ref{iota}(1). Since the involutions $\iota$ and $\theta$ commute, $\iota\circ\theta$ is a nontrivial involution of the K3 surface $S$ preserving $\omega_S$. Thus, $S^{\iota\circ\theta}$ consists of exactly eight points by the fundamental result due to Nikulin \cite{Ni80}, see also \cite[(0.1)]{Mu88}. On the other hand, according to Lemma \ref{iota}(2), $\iota$ acts as a nontrivial involution on the four disjoint smooth rational curves $C$, $F_1$, $F_3$, $E_4$. Then each of the eight distinct points $C^\iota\cup F_1^\iota\cup F_3^\iota\cup E_4^\iota$ is fixed by both $\theta$ and $\iota$ by Lemma \ref{theta}(1). From these observations, we conclude that 
$$C^\iota\cup F_1^\iota\cup F_3^\iota\cup E_4^\iota=S^\iota\cap S^\theta=S^{\iota\circ\theta}=\{x\in S |\, \iota(x)=\theta(x) \}.$$ 
Thus, by Lemma \ref{iota}(1)(3), the fixed locus $T^{\iota_{n, T}}$ of $\iota_{n, T} \in \Aut (T)$ is equal to the disjoint union of the four smooth irreducible curves 
$$C_{11, T}=\pi(f_1^{-n}(C_{11}))\,\, ,\,\, \pi(f_1^{-n}(C_{31}))\,\, ,\,\, \pi(f_1^{-n}(C_{34}))\,\, ,\,\, \pi(f_1^{-n}(\Sigma_0)),$$ 
whose self-intersection numbers are $-1$, $-1$, $-1$, $3$ because the self-intersection numbers of 
$C_{11}, C_{31}, C_{34}, \Sigma_0$ are $-2, -2, -2, 6$
respectively. 

Since $\Sigma_0$ is a nef and big divisor on $S$, so is $\pi(f_1^{-n}(\Sigma_0))$ on $T$. It follows that there are only finitely many irreducible curves $D_{n1},\ldots,D_{nk}$ for some positive integer $k\ge 3$ such that $(D_{ni}.\pi(f_1^{-n}(\Sigma_0)
))_{T}=0$ for $i=1,\ldots,k$. We may and will assume that $D_{n1}= C_{11, T}$, $D_{n2}=\pi(f_1^{-n}(C_{31}))$,  $D_{n3}=\pi(f_1^{-n}(C_{34}))$. 
Observe that all these curves have negative self-intersection numbers. Therefore, except for $i=1$, they belong to the family $\sA$ defined above.
This, together with Lemma \ref{pointsinC11}, imply 

\begin{enumerate}
\item[(Q1)] $\Aut(S, \{Q, \theta(Q)\})=\Aut (S, C_{11}, \{Q, \theta(Q)\})=\Aut (S, C_{11}, \{P, P_1, Q, \theta(Q)\})$;
\item[(Q2)] $Q_{T} := \pi(Q) \notin D_{ni}\cap C_{11, T}$ for any $n$ and any $i\ge 2$.
\end{enumerate}

Recall that $X = X_{Q_T}$ is the blow up of $T$ at the point $Q_{T} \in C_{11, T}(\R)$. Let $E_{Q_T} \subset X$ be the exceptional curve over $Q_{T}$. We may and will identify $\Aut(T, Q_{T})$ as a subgroup of $\Aut (X)$ in a natural manner. 
Suppose $g\in \Aut (X)$ and $g^{-1}\circ \iota_{n,X}\circ g= \iota_{m, X}$ for some $m,n\geq 0$. We first show that $g\in \Aut(T, Q_{T})$. 

Since $\iota_{n,T}$ is a nontrivial involution and $C_{11,T}$ is contained in its fixed locus, $\iota_{n, X}$ acts as a nontrivial involution on $E_{Q_T}$. Therefore, it fixes exactly two points of $E_{Q_T}$, say $P_{n1}$ and $P_{n2}$. We may assume that $P_{n1}= C_{11, X}\cap E_{Q_T}$. Here, $C_{11, X}$ is the proper transform of $C_{11, T}$ in $X$. Then the fixed point locus of $\iota_{n, X}\in \Aut (X)$ (resp. $\iota_{m, X}\in \Aut (X)$) is the disjoint union of 
the point $P_{n2}$ (resp. $P_{m2}$) with four smooth irreducible curves, of self-intersection numbers $-2,-1,-1,3$,  which are the proper transforms of the fixed locus of $\iota_{n,T}$ (resp. $\iota_{m,T}$). Since  $g^{-1}\circ \iota_{n,X}\circ g= \iota_{m, X}$, we have
$$g(X^{\iota_{m, X}})=X^{\iota_{n, X}}.$$
It follows that 
$$g(C_{11, X})= C_{11, X},\ g(\pi(f_1^{-m}(\Sigma_0))^X)=\pi(f_1^{-n}(\Sigma_0))^X,\ g(P_{m2})=P_{n2},$$
where we add the superscript $X$ to denote the proper transform of a curve to $X$.
Note that the irreducible curves in $X$ which do not intersect $\pi(f_1^{-n}(\Sigma_0))^X$ (resp. $\pi(f_1^{-m}(\Sigma_0))^X$) are exactly $E_{Q_T}, D_{n1}^X,\ldots,D_{nk}^X$ (resp. $E_{Q_T}, D_{m1}^X,\ldots,D_{mk}^X$). Thus, 
$$g(E_{Q_T} \cup D_{m1}^X\cup\ldots\cup D_{mk}^X)=E_{Q_T} \cup D_{n1}^X\cup \ldots\cup D_{nk}^X.$$
 Then by Property (Q2) above, we have that 
 $E_{Q_T}$ doesn't intersect $D_{mi}^X$ and $D_{ni}^X$. Using that $g(P_{m2})=P_{n2}$, we obtain
 $g(E_{Q_T})=E_{Q_T}$. Therefore, $g\in \Aut(T, Q_{T})$. 

Using Property (Q1) and Lemma \ref{theta}(3), we have 
$$\Aut(T, Q_{T})=\Aut(S, \{Q, \theta(Q)\})/\langle \theta \rangle=\Aut(S,C_{11},\{P,P_1,Q,\theta(Q)\})/\langle \theta \rangle.$$ 
Then by Lemma \ref{infiniteconj},  the conjugacy classes of $\{\iota_{n, X}\}$ in the group  $\Aut(T, Q_T)$ is infinite. Since $\iota_{n, X}$ cannot be conjugated to $\iota_{m, X}$ by elements in $ \Aut (X)\setminus  \Aut(T, Q_T)$ as observed above, the conjugacy classes of $\{\iota_{n, X}\}$ in the group $\Aut (X)$ is infinite too. This completes the proof.
\end{proof}

\section{Proof of Theorem \ref{Main2Intro}}\label{sect4}

We continue to use the same notations as in Section~\ref{sect2}. For instance, $\lambda \in \R$ is generic and 
$$P = C \cap C_{11} \in C\, ,\, C_{11} \subset S = {\rm Km}\, (E \times F)\,  ,\, f_1, \iota \in \Aut (S)$$
are the same as in Section \ref{sect2} and all defined over $\R$ with respect to the natural real structure of the elliptic curves $E$ and $F$.

In this section, we prove the following theorem, from which Theorem \ref{Main2Intro} clearly follows. 

\begin{theorem}\label{Mukai}
There exists a point $Q\in C_{11}(\R)$ such that the blow up $S'$ of $S$ at $Q$ satisfies the following two statements:
\begin{enumerate}
\item $\Aut (S')$ is discrete and not finitely generated; and
\item $S'$ admits infinitely many real forms which are mutually non-isomorphic over $\R$.
\end{enumerate}
\end{theorem}

\begin{proof}
Note that $C_{11}(\R)$ is an uncountable set. By Lemma \ref{specialpoint}, there exists a point $Q\in C_{11}(\R)\setminus (C\cup F_1)$ such that 
$$\Aut (S, Q) = \Aut (S, C_{11}, Q).$$
Let $S'$ be the blow up of $S$ at $Q$ and $E_Q \subset S'$ be the exceptional divisor. Then 
$$\Aut (S,C_{11}, Q)=\Aut (S, Q)=\Aut (S').$$
Note that $|K_{S'}| = \{E_Q\}$ by the canonical bundle formula. The last equality follows from this. 

Recall the affine coordinate $x$ of $C$ in Section \ref{sect2} and the actions of $f_1$, $\iota$ on $C$ described as in Lemma \ref{iota}. In particular, we have 
$$\iota|_{C}(x) = 2 -x\,\, ,\,\, \iota_n|_{C}(x)=(f_1^{-n}\circ \iota \circ f_1^n)(x)=\frac{1}{2^{n-1}}-x.$$
Let $f_3:=\iota \circ\iota_1$. Then $f_3 \in \Aut (S, C)$ and we obtain
$$f_3^*\omega_{S} = \omega_{S}
 \quad \text{and} \quad f_3|_C(x)=x+1.$$
 By Lemma \ref{ine}, we have the representation 
$$\tau: \Ine^s (S,C_{11})\longrightarrow \Aut (C,P)$$ 
and the image $\tau( \Ine^s (S,C_{11}))$ is an abelian group. On the other hand, since $f_1^{-n}\circ f_3\circ f_1^n\in \Ine^s (S,C_{11})$ and 
$$f_1^{-n}\circ f_3\circ f_1^n|_{C} (x) = x + \frac{1}{2^n},$$ 
it follows that 
the {\it abelian} group $\tau( \Ine^s (S,C_{11}))$ contains a non-finitely generated abelian group as a subgroup, see e.g. \cite[Proposition\, 2.5(2)]{DO19}. Thus, $ \Ine^s (S,C_{11})$ is not finitely generated. Consider the sequence of finite index subgroups (see the end of the proof of Lemma \ref{infiniteconj} for the first inclusion)
$$\Ine^{s}(S,C_{11}) \subset \Ine (S,C_{11}) \subset \Aut(S, C_{11}, \{P,P_1,Q\})=\Aut (S,C_{11}, Q)= \Aut (S, Q),$$
and the equality
$$\Aut (S, Q) = \Aut (S')$$
under the natural identification made above (using the fact that $S$ is a smooth K3 surface). Therefore, $\Aut (S')$ is not finitely generated, as it has a finite index subgroup $\Ine^{s}(S,C_{11})$ which is not finitely generated, see e.g. \cite[Proposition\, 2.5(1)]{DO19}. Hence (1) is proved. 

Since $Q\in C_{11}(\R)\subset S(\R)$, it follows that $S'$ is defined over $\R$.  Consider the group
$$\Ine^{\pm}(S,C_{11}) := \{g\in \Ine(S, C_{11}) |\, g^*(\omega_S)=\pm \omega_S\}$$ 
defined in Section \ref{sect2}. By Lemma \ref{ine}, $\iota_n\in \Ine^{\pm}(S,C_{11})$ and $\Ine^{\pm}(S,C_{11})$ is a finite index subgroup of $\Aut (S, C_{11}, \{P,P_1,Q\}) = \Aut(S')$ (under the natural identification made above). By Lemma \ref{infiniteconj}, $\iota_n$ define infinitely many conjugacy classes in $\Ine^{\pm}(S,C_{11})$. Note that ${\rm Gal}(\C/\R)$ acts trivially on $\Ine^{\pm}(S,C_{11})$, see e.g. \cite[Lemma 4.6]{DO19}. Then by \cite[Lemma 13]{Le18}, $S'$ admits infinitely many real forms which are mutually non-isomorphic over $\R$.
\end{proof}

\end{document}